\documentclass[12pt]{amsart}

\usepackage{amssymb}

\usepackage[bottom]{footmisc}

\usepackage{hyperref}

\usepackage{booktabs}

\usepackage{cite}

\numberwithin{equation}{section}

\theoremstyle{break}
 \newtheorem{theorem}{Theorem}[section]
 \newtheorem{proposition}[theorem]{Proposition}
 \newtheorem{corollary}[theorem]{Corollary}
 \newtheorem{lemma}[theorem]{Lemma}

\allowdisplaybreaks[4]

\def\C{\mathcal C}
\def\cg{\C(G)}
\def\cgn{\C(G,N)}
\def\cgq{\C(G/N)}
\def\qq{{\mathbb Q}}
\def\ovk{\overline{K}}
\def\ff{{\mathbb F}}
\def\p{{\mathcal P}}

\def\de{\Delta}
\def\Aut{{\mathsf A}{\mathsf u}{\mathsf t}}

\newcommand\blfootnote[1]{
    \begingroup
    \renewcommand\thefootnote{}\footnote{#1}
    \addtocounter{footnote}{-1}
    \endgroup
}

\begin{document}

\title[Invariable generation and rational homology]{Invariable generation of finite simple groups and rational homology of coset posets}

\author[R. M. Guralnick]{Robert M. Guralnick}
\address{Department of Mathematics, University of Southern California, 3620 S. Vermont Ave, Los Angeles, CA 90089-2532 USA}
\email{guralnic@usc.edu}

\author[J. Shareshian]{John Shareshian}
\address{Department of Mathematics, Washington University, One Brookings Drive, St Louis, MO 63124 USA}
\email{jshareshian@wustl.edu}

\author[R. Woodroofe]{Russ Woodroofe}
\address{Univerza na Primorskem, Glagolja\u{s}ka 8, 6000 Koper, Slovenia}
\email{russ.woodroofe@famnit.upr.si}
\urladdr{\url{https://osebje.famnit.upr.si/~russ.woodroofe/}}

\begin{abstract}
We show that every finite simple group is generated invariably by a Sylow subgroup and a cyclic group. It follows that that the order complex of the coset poset of an arbitrary finite group has nontrivial reduced rational homology.
\end{abstract}

\maketitle

\blfootnote{Work of the first author was partially supported by NSF Grant DMS-1901595 and Simons Foundation
Fellowship 609771. Work of the second author was supported in part by NSF Grant DMS 1518389. Work
of the third author is supported in part by the Slovenian Research Agency research program P1-0285 and
research projects J1-9108, N1-0160, J1-2451, J1-3003, and J1-50000.}

\section{Introduction}

Given a group $G$ and subsets $S,T$ of $G$, we say that $S$ and $T$ {\it generate $G$ invariably} if 
$$
\langle S^g,T^h \rangle=G
$$
for all $g,h \in G$.  Often, we identify a set of size one with the element that it contains.  We are interested in invariable generation of finite simple groups.  Herein we prove the following result, using the Classification.

\begin{theorem} \label{main1}
Every finite simple group is generated invariably by a Sylow subgroup and a cyclic subgroup.
\end{theorem}

We apply Theorem~\ref{main1} to give a streamlined proof of a stronger result than one proved by two of the present authors in \cite{Shareshian/Woodroofe:2016}.  Given a finite poset $\p$, the {\it order complex} $\Delta \p$ is the abstract simplicial complex whose $d$-dimensional faces are the chains of length $d$ (cardinality $d+1$) from $\p$.  Given a finite group $G$, we write $\cg$ for the set of all cosets of all proper subgroups of $G$ (including the trivial subgroup if $|G|>1$), ordered by inclusion.  The main result in \cite{Shareshian/Woodroofe:2016} is that if $G$ is a finite group, then $\Delta \cg$ has nontrivial reduced homology in characteristic~$2$, and is therefore not contractible.  This settles a question raised by Brown in \cite{Brown:2000}.
In Section \ref{sec.cosetposet}, we use Theorem~\ref{main1} and Smith Theory to prove the following result.

\begin{theorem} \label{cosetposet}
If $G$ is a finite group, then $\Delta \cg$ has nontrivial reduced rational homology.
\end{theorem}

The above-mentioned characteristic two result follows from Theorem~\ref{cosetposet} and the Universal Coefficient Theorem.  Given the results in \cite{Brown:2000} and \cite{Shareshian/Woodroofe:2016}, it is natural to ask whether the reduced Euler characteristic $\widetilde{\chi}(\Delta \cg)$ is nonzero for all finite groups $G$.  This question remains open.

In \cite{Damian/Lucchini:2007}, Damian and Lucchini showed that ``most" finite simple groups are generated invariably by a Sylow $2$-subgroup and a cyclic subgroup of prime order.  Since there is no constraint on the prime dividing the order of the Sylow subgroup in Theorem~\ref{main1}, we may take an easier route than that taken in \cite{Damian/Lucchini:2007}.  It is clear that Theorem~\ref{main1} holds for groups of prime order.  In Sections \ref{sec.alternating} and \ref{sec.sporadic}, we confirm results stronger than Theorem~\ref{main1} for alternating groups and sporadic groups.

\begin{proposition} \label{alternating}
For each integer $n \geq 5$, the alternating group $A_n$ is generated invariably by two elements, at least one of which has prime order.
\end{proposition}
 
 \begin{proposition} \label{sporadic}
 All sporadic simple groups except the Mathieu groups $M_{11}$, $M_{12}$, and $M_{24}$ are generated invariably by two elements of prime order.  Moreover,
 \begin{itemize}
 \item $M_{11}$ is generated invariably by an element of order eleven and an element of order eight;  
 \item $M_{24}$ is generated invariably by an element of order $23$ and an element of order four; and  
 \item $M_{12}$ is generated invariably by an element of order eleven and an element of order ten, but not by two elements of prime power order.
 \end{itemize}
\end{proposition}

For simple groups of Lie type, we have a result slightly stronger than Theorem~\ref{main1}.

\begin{proposition} \label{Lietype}
If $K$ is a simple group of Lie type in characteristic~$p$, and not isomorphic with one of $L_6(2)$, $U_4(2)$, $Sp_6(2)$, or $\Omega_8^+(2)$, then $K$ is generated invariably by a Sylow $p$-subgroup and an element of prime order.
\end{proposition}

We will show in Lemma~\ref{l62} below that $L_6(2)$ is generated invariably by an element of order seven and an element of order~$31$.  Consulting \cite{Conway/Curtis/Norton/Parker/Wilson:1985}, we see that $Sp_6(2)$ is generated invariably by a Sylow $3$-subgroup and an element of order seven and that $O_8^+(2)$ is generated invariably by a Sylow $3$-subgroup and an element of order five. The group $U_4(2)$ is isomorphic with $PSp_4(3)$ and, as we shall prove, is generated invariably  by a Sylow $3$-subgroup and an element of order five.  Hence, Proposition~\ref{Lietype} suffices to complete the proof of Theorem~\ref{main1}.  We prove Proposition~\ref{Lietype} in Section \ref{sec.Lietype}.  The key point is that in almost all cases there is a Zsigmondy prime for some $q^e-1$ that divides $|K|$ but does not divide the order of any parabolic subgroup.  

In subsequent work we will show that all but a negligible portion of the finite simple groups of Lie type are generated invariably by an element of prime order and an element of prime power order, and ``most" of these are generated invariably by two elements of prime order.

After proving our results in Sections \ref{sec.alternating}, \ref{sec.sporadic}, \ref{sec.Lietype}, and \ref{sec.cosetposet}, we make some final comments in Section \ref{sec.final}.

\section{Alternating groups} \label{sec.alternating}

We prove first a more specific version of Proposition~\ref{alternating} that holds for all $n>7$.  Bertrand's Postulate, which was proved by Chebyshev in \cite{Chebyshev:1852}, says that the prime $p$ appearing in Proposition~\ref{alt2} exists.

\begin{proposition} \label{alt2}
Assume that $n>7$.  Let $p$ be a prime satisfying $\frac{n}{2}<p<n-2$ and let $y \in A_n$ have order $p$.  If $n$ is odd let $x \in A_n$ be an $n$-cycle.  If $n$ is even let $x \in A_n$ be the product of two disjoint $\frac{n}{2}$-cycles.  Then $x$ and $y$ generate $A_n$ invariably. 
\end{proposition}

\begin{proof}
As $p>\frac{n}{2}$, each $\langle x \rangle$-orbit intersects the support of $y$ nontrivially, hence $\langle x,y \rangle$ is transitive.  No imprimitive subgroup contains $y$, since every prime divisor of the order of such a subgroup is at most some proper divisor of $n$.  So $\langle x,y \rangle$ is primitive.  Finally, a result of Jordan (see \cite{Jordan:1875} or \cite[Theorem 3.3E]{Dixon/Mortimer:1996}) shows that no primitive proper subgroup of $A_n$ contains $y$, since $y$ fixes more than two points.  Therefore $\langle x,y \rangle=A_n$.  As every conjugate of $x$ has the same cycle type as $x$ and the same goes for $y$, we can apply the same argument to conclude that $\langle x^g,y^h \rangle=A_n$ for all $g,h \in A_n$.
\end{proof}

Inspection shows that $A_5$ is generated invariably by a $5$-cycle and a $3$-cycle, $A_6$ is generated invariably by a $5$-cycle and an element of cycle type $(4,2)$, and $A_7$ is generated invariably by a $7$-cycle and a $5$-cycle.  (One can use \cite{Conway/Curtis/Norton/Parker/Wilson:1985}.)  This completes the proof of Proposition~\ref{alternating}.



\section{Sporadic groups} \label{sec.sporadic}

We prove Proposition~\ref{sporadic} using facts in \cite{Conway/Curtis/Norton/Parker/Wilson:1985} along with the updates given in \cite{Wilson:2017}.  For each sporadic group $S$ listed in Table \ref{tab:SporadicGroups}, we give element orders $p,r$ so that $pr$ divides $\left|S\right|$, but does not divide $\left|M\right|$ for $M<S$. Thus, these sporadic groups are generated invariably by elements of these orders. The orders $p,r$ are prime, except when $S=M_{11}$, where we take $r=8$; it can be seen in  \cite{Conway/Curtis/Norton/Parker/Wilson:1985} that $M_{11}$ indeed has an element of order eight.

\begin{table}

\caption{\label{tab:SporadicGroups}Sporadic groups $S$ and $p,r$ dividing $|S|$ such that no proper subgroup has order divisible by $pr$}

\begin{tabular}{rllrllrl}
\toprule$S$ & $p,r$ & $\qquad$ & $S$ & $p,r$ & $\qquad$ & $S$ & $p,r$\tabularnewline
\cmidrule{1-2} \cmidrule{2-2} \cmidrule{4-5} \cmidrule{5-5} \cmidrule{7-8} \cmidrule{8-8} 
$M_{11}$ & $11,8$ &  & $Co_{1}$ & $23,13$ &  & $ON$ & $31,19$\tabularnewline
$M_{22}$ & $11,7$ &  & $Fi_{22}$ & $13,11$ &  & $HN$ & $19,11$\tabularnewline
$M_{23}$ & $23,7$ &  & $Fi_{23}$ & $23,17$ &  & $Ly$ & $67,37$\tabularnewline
$J_{1}$ & $19,11$ &  & $Fi'_{24}$ & $29,23$ &  & $Th$ & $31,19$\tabularnewline
$J_{2}$ & $7,5$ &  & $He$ & $17,7$ &  & $B$ & $47,31$\tabularnewline
$J_{3}$ & $19,17$ &  & $Ru$ & $29,13$ &  & $M$ & $71,59$\tabularnewline
$J_{4}$ & $43,37$ &  & $Suz$ & $13,11$ &  &  & \tabularnewline
\bottomrule
\end{tabular}
\medskip
\end{table}


\medskip
It remains to show that $M_{11}$ is not generated invariably by two elements of prime order, and to examine $M_{12}$, $M_{24}$, $HS$, $Co_2$, $Co_3$, and $McL$. We use facts and notation appearing in \cite{Conway/Curtis/Norton/Parker/Wilson:1985}.

The primes dividing $|M_{11}|$ are $2,3,5,11$.  Given any such prime, $M_{11}$ contains one conjugacy class of elements of the given order.  As $L_2(11)$ embeds in $M_{11}$ and has order divisible by all of $2,3,5,11$, we see that our claim about $M_{11}$ holds.

The conjugacy classes of nontrivial elements of prime power order in $M_{12}$ are $2A$, $2B$, $3A$, $3B$, $4A$, $4B$, $5A$, $8A$, $8B$, $11A$ and $11B$.  Each element of order four is the square of an element of order eight.  So, classes $4A$ and $4B$ can be ignored.  Elements in class $2B$ are squares of elements of order four.  So, we can also ignore this class.  Elements in class $11A$ are powers of elements of class $11B$, and vice versa.  So, we may consider these two classes as one, which we call class $11AB$.  Classes $8A$ and $8B$ fuse in $\Aut(M_{12})$.  As no two classes consisting of elements of odd order fuse in $\Aut(M_{12})$, we may consider these two classes as one, which we call class $8AB$.

By Sylow's Theorem, we may ignore the pairs of classes $(2A,8AB)$ and $(3A,3B)$.  A point stabilizer in the natural action of $M_{12}$ on twelve points is isomorphic to $M_{11}$, and thus contains elements of orders $2,3,5,8,11$.  The group $L_2(11)$ embeds in $M_{12}$ through its action on the set of twelve $1$-spaces from $\ff_{11}^2$.
The image of this embedding contains elements of orders $2,3,5$, and $11$; those of orders two and three act without fixed points.  An element of $M_{12}$ having order three and fixing a point fixes three points and thus stabilizes a $2$-set in the natural action, as does every element of order two.  The stabilizer of a $3$-set in the natural action is isomorphic with $AGL_2(3)$ and thus contains an element of order eight, along with elements of classes $3A$ and $3B$.  It is now straightforward to confirm that $M_{12}$ is not generated invariably by two elements of prime power order.  

A maximal subgroup of $M_{12}$ having order divisible by eleven is isomorphic to one of $M_{11}$ or $L_2(11)$, and therefore contains no element of order ten.  On the other hand, $M_{12}$ contains elements of order ten.  Any such element generates $M_{12}$ invariably with any element of order eleven.

If $g \in M_{24}$ has order $23$ and $H<M_{24}$ is a maximal subgroup containing $g$, then $H$ belongs to either the unique conjugacy class of subgroups isomorphic to $M_{23}$ or the unique conjugacy class of maximal subgroups isomorphic to $L_2(23)$.  Each of $M_{23}$ and $L_2(23)$ contains a unique conjugacy class of elements of order four, while $M_{24}$ contains three classes of elements of order four.  So, some element of order four and any element of order $23$ generate $M_{24}$ invariably.  (One can show that such an element of order four lies in class $4A$.)  Using arguments similar to those employed when examining $M_{12}$, one can show that if $C_1,C_2$ are conjugacy classes in $M_{24}$ consisting of elements of prime order, then there exists $(c_1,c_2) \in C_1 \times C_2$ such that either  $\langle c_1,c_2 \rangle$ stabilizes a point, a $2$-set or a $3$-set in the natural action of $M_{24}$ on twenty four points, or $\langle c_1,c_2 \rangle$ is contained in a maximal subgroup isomorphic with $L_2(23)$.

Assume now that the group $S$ is one of $HS$ or $McL$.  Then $|S|$ is divisible by eleven.  If $H$ is a maximal subgroup of $S$ with $|H|$ divisible by eleven, then $H$ is isomorphic with one of the Mathieu groups $M_{11}$ or $M_{22}$.  It follows that $H$ has a unique conjugacy class of elements of order five, and every element of this class normalizes a subgroup of order eleven.  An element of order eleven in $S$ generates its own centralizer.  So, if $K<S$ has order eleven, then $|N_S(K)|$ is not divisible by $25$.  As $121$ does not divide $|S|$, we know that all subgroups of order eleven in $S$ are conjugate.  It follows now that there is a unique conjugacy class $C$ of elements of order five in $S$ such that an element of $C$ and any element of order eleven do not generate $S$ invariably.  Indeed, $C$ consists of those elements of order five in $S$ that normalize a subgroup of order eleven.  As each of $HS$ and $McL$ has two conjugacy classes of elements of order five, our claim about these groups holds.

Finally, assume that $S$ is one of $Co_2$ or $Co_3$.  Every maximal subgroup of $S$ containing an element of order $23$ lies in one conjugacy class of groups isomorphic to $M_{23}$.  Now $M_{23}$ contains  a unique conjugacy class of elements of order two, while $Co_2$ contains three such classes and $Co_3$ contains two such classes.  Our claim about these groups follows.

\section{Groups of Lie type} \label{sec.Lietype} We call a finite simple group $K$ a group of Lie type in characteristic~$p$ if there exist a simple algebraic group $\ovk$ over the algebraically closed field $\overline{{\mathbb F}_p}$, and an endomorphism $\sigma$ of $\ovk$, such that $K$ is generated by the elements of order $p$ in the fixed-point group $\ovk^\sigma$.  (Almost always $K=\ovk^\sigma$.)  A {\it parabolic subgroup} of $K$ is any subgroup containing the normalizer of a Sylow $p$-subgroup of $K$.  A result of Tits says that every maximal subgroup of $K$ containing a Sylow $p$-subgroup of $K$ is parabolic.  (See \cite[1.6]{Seitz:1973}.)  We seek a prime $r$ dividing $|K|$ but not dividing the order of any parabolic subgroup of $K$, as in this case $K$ is generated invariably by a Sylow $p$-subgroup and an element of order $r$. 

\subsection{An example for the uninitiated - $GL_n(q)$ and $PSL_n(q)$} The reader unfamiliar with finite groups of Lie type might benefit from considering the groups $GL_n(q)$, in which parabolic subgroups are stabilizers of nontrivial proper subspaces of the natural vector space.  If one identifies $\ff_q^n$ with $\ff_{q^n}$ and considers the (linear) action of a multiplicative generator $\alpha$ of $\ff_{q^n}^\ast$, one obtains a cyclic subgroup $C \leq GL_n(q)$ having order $q^n-1$.  On the other hand, the stabilizer $P_k$ of a $k$-dimensional subspace of $\ff_q^n$ satisfies 
$$
|P_k|=q^{{n} \choose {2}}\prod_{j=1}^k(q^j-1)\prod_{j=1}^{n-k}(q^j-1).
$$
Hence, $C$ acts irreducibly on $\ff_q^n$, as does any subgroup $Q \leq C$ whose order is a prime dividing $q^n-1$ but not dividing $q^j-1$ for $j<n$.  (Such primes are the Zsigmondy primes discussed in Section \ref{zsig} below.)  It turns out that $Q \leq SL_n(q)$ and the image of $Q$ in $K=PSL_n(q)$ will generate $K$ invariably along with any Sylow $p$-subgroup, where $p$ divides $q$.  The subgroup $C$ is often called a Coxeter torus, and is equal to $\overline{T}^\sigma$ for some maximal torus $\overline{T} \leq \overline{K}:=GL_n(\overline{\ff_p})$, where $\sigma$ acts on $\overline{K}$ by raising every matrix entry to its $q^{th}$ power.  In the general case where $K=\overline{K}^\sigma$, almost always there is some maximal torus $\overline{T} \leq \overline{K}$ such that $\overline{T}^\sigma$ has order divisible by some prime $r$ that does not divide the order of any parabolic subgroup of $K$.  It turns out that in this general case we can take $r$ to be a well-chosen Zsigmondy prime.

\subsection{Zsigmondy's Theorem} \label{zsig} We recall now a version of a theorem of Zsigmondy that suits our purposes.  Given positive integers $q$ and $e$, a {\it Zsigmondy prime} for the pair $(q,e)$ is a prime $r$ dividing $q^e-1$ but not dividing $q^f-1$ if $f<e$ is a positive integer.

\begin{theorem}[Zsigmondy, see \cite{Zsigmondy:1892}] \label{zsigthm}
A Zsigmondy prime for $(q,e)$ exists unless one of
\begin{itemize}
\item $e=6$ and $q=2$, or
\item $e=2$ and $q=2^k-1$ for some integer $k$
\end{itemize}
holds.
\end{theorem}

\subsection{An example for the initiated - $E_8(q)$} We use notation and terminology from \cite[Chapter 2]{Gorenstein/Lyons/Solomon:1998}.  \label{e8q} Given a power $q$ of a prime $p$, the simple group $K=E_8(q)$ is untwisted of universal type.  It follows from \cite[Theorem 2.6.5]{Gorenstein/Lyons/Solomon:1998} that a parabolic subgroup of $K$ is a semidirect product $P=UL$, where $U$ is a $p$-group and $L$ is a product $HM$. Here $M \unlhd HM$ is isomorphic to one of $D_7(q)$, $A_1(q) \times A_6(q)$, $A_4(q) \times A_2(q) \times A_1(q)$, $A_4(q) \times A_3(q)$, $D_5(q) \times A_2(q)$, $E_6(q) \times A_1(q)$, or $E_7(q)$, all components in these products being of universal type.  As $K$ is untwisted, the Cartan subgroup $H$ is isomorphic with $({\mathbb Z}_{q-1})^8$, by \cite[Table 2.4 and Theorems 2.4.1 and 2.4.7]{Gorenstein/Lyons/Solomon:1998}.  Inspection shows that a Zsigmondy prime for $(q,30)$ divides $|K|$ but does not divide $|P|$.  The reader familiar with the theory of finite groups of Lie type will have observed that a similar argument shows that whenever $K$ is an untwisted finite simple group of Lie type, one can take $e$ to be the highest exponent of the Weyl group. This is done in Table \ref{tab:Lietype}.  The twisted case is more complicated, as one sees upon considering $^2A_n(q)$ with $n$ odd.

\subsection{The general case} Using the notation for finite groups of Lie type found in \cite[Table 2.2 on page 39]{Gorenstein/Lyons/Solomon:1998}, we list in Table \ref{tab:Lietype} each family $\{^d\Sigma^\epsilon(q)\}$ of such groups.  For each such family, we give an exponent $e$ such that a Zsigmondy prime for $(q,e)$ divides the order of $K=^d\Sigma^\epsilon(q)$ but does not divide the order of any parabolic subgroup of $K$.  The order of $K$ is given in \cite[Table 2.2]{Gorenstein/Lyons/Solomon:1998}.  In the last column of our table one finds for each family a reference from which the orders of parabolic subgroups can be determined.  


\begin{table}

\caption{\label{tab:Lietype}Exponents for Zsigmondy primes that avoid parabolic subgroups of the various groups
of Lie type}

\begin{tabular}{rcl}
\toprule Lie type of $K$ & $e$ & Reference\tabularnewline
\midrule$A_{n}^{+}(q)$ & $n+1$ & \cite[Proposition 4.1.17]{Kleidman/Liebeck:1990}\tabularnewline
$A_{n}^{-}(q)$, $n$ even & $2(n+1)$ & \cite[Proposition 4.1.18]{Kleidman/Liebeck:1990}\tabularnewline 
$A_{n}^{-}(q)$, $n\geq3$ odd & $2n$ & \cite[Proposition 4.1.18]{Kleidman/Liebeck:1990}\tabularnewline 
$B_{n}(q)$, $n\geq3$ & $2n$ & \cite[Proposition 4.1.20]{Kleidman/Liebeck:1990}\tabularnewline 
$C_{n}(q)$, $n\geq3$ & $2n$ & \cite[Proposition 4.1.19]{Kleidman/Liebeck:1990}\tabularnewline
$D_{n}^{+}(q)$, $n\geq4$ & $2n-2$ & \cite[Proposition 4.1.20]{Kleidman/Liebeck:1990}\tabularnewline
$D_{n}^{-}(q)$, $n\geq4$ & $2n$ & \cite[Proposition 4.1.20]{Kleidman/Liebeck:1990}\tabularnewline
$^{3}D_{4}(q)$ & $12$ & \cite{Kleidman:1988b}\tabularnewline
$G_{2}(q)$ & $6$ & \cite{Cooperstein:1981},\cite{Kleidman:1988a}\tabularnewline
$F_{4}(q)$ & $12$ & \cite{Craven:2023}\tabularnewline
$E_{6}^{+}(q)$ & $12$ & \cite{Craven:2023}\tabularnewline
$E_{6}^{-}(q)$ & $12$ & \cite{Craven:2023}\tabularnewline
$E_{7}(q)$ & $18$ & \cite{Craven:2022UNP}\tabularnewline 
$E_{8}(q)$ & $30$ & See Section \ref{e8q}\tabularnewline
$^{2}B_{2}(q)$, $2|q$ & $4$ & \cite{Suzuki:1962}\tabularnewline
$^{2}F_{4}(q)$, $2|q$ & $12$ & \cite{Malle:1991}\tabularnewline 
$^{2}G_{2}(q)$, $3|q$ & $6$ & \cite{Kleidman:1988a}\tabularnewline 
\bottomrule 
\end{tabular}
\medskip
\end{table}


\medskip
We are left with the cases where there is no Zsigmondy prime for $(q,e)$.  By Theorem~\ref{zsigthm} and inspection, this occurs only if $K=L_2(p)$ for some Mersenne prime $p$  or the type of $K$ is one of $G_2(2)$, $A_5^+(2)$, $A_2^-(2)$, $A_3^-(2)$, $B_3(2)$, $C_3(2)$, or $D_4^+(2)$.  (Here we use Catalan's Conjecture, which says that if $a$ and $a+1$ are both perfect powers then $a=8$ and which was proved by Mih\u{a}ilescu in \cite{Mihuailescu:2004}.)

The only parabolic subgroups of $K=L_2(p)$ are the Borel subgroups, isomorphic with the semidirect product ${\mathbb Z}_p:{\mathbb Z}_{(p-1)/2}$.  If $p$ is a Mersenne prime then $\frac{p-1}{2}$ is odd, hence a nontrivial involution in $K$ is contained in no parabolic subgroup.

The simple group of type $G_2(2)$ is generated invariably by a Sylow $2$-subgroup and an element of order seven, as can be seen in \cite{Conway/Curtis/Norton/Parker/Wilson:1985} or \cite{Cooperstein:1981}. We remark that this group is isomorphic with that of type $A_2^-(3)$ and has been handled already.  There is no simple group of type $A_2^-(2)$, by an order argument. The remaining groups are the classical groups $L_6(2)$, $U_4(2)$, $\Omega_7(2)$, $Sp_6(2)$, and $\Omega_8^+(2)$.  We observe that $L_6(2)$, $U_4(2)$, $Sp_6(2)$ and $\Omega_8^+(2)$ are the exceptions listed in Proposition~\ref{Lietype}. Moreover, $\Omega_7(2) \cong Sp_6(2)$, as noted in \cite[Theorem 2.2.10]{Gorenstein/Lyons/Solomon:1998}.  

The next lemma, along with the discussion immediately following Proposition~\ref{Lietype}, completes our proof of Theorem~\ref{main1}.

\begin{lemma} \label{l62}
The group $L_6(2)$ is generated invariably by an element of order $31$ and an element of order seven.
\end{lemma}

\begin{proof}
A complete list of maximal subgroups of $SL_6(q)$ appears in \cite[Tables 8.24 and 8.25]{Bray/Holt/Roney-Dougal:2013}.  Consulting this list, and observing that $L_6(2) \cong SL_6(2)$, we see that the maximal subgroups of $L_6(2)$ having order divisible by $31$ are the stabilizers of $1$-dimensional subspaces and the stabilizers of $5$-dimensional subspaces in the natural action on $\ff_2^6$.  Since seven is a Zsigmondy prime for $(2,3)$, the group ${\mathbb Z}_7$ has an irreducible representation of degree three over $\ff_2$.  Therefore, there is an element of order seven in $L_6(2)$ that stabilizes some $3$-dimensional subspaces but no other subspaces of $\ff_2^6$.  Such an element must generate $L_6(2)$ invariably with an element of order $31$.
\end{proof}

\section{Proof of Theorem~\ref{cosetposet}} \label{sec.cosetposet}

Our proof of Theorem~\ref{cosetposet} proceeds along the same lines as that of the analogous characteristic two result given in \cite{Shareshian/Woodroofe:2016}.  Much of what we write here will be less detailed than what appears in \cite{Shareshian/Woodroofe:2016}, which the interested reader may consult for extended arguments.  

\subsection{Smith Theory} We review here results from Smith (or Smith-Oliver) Theory that are stronger than what we will need.  This theory describes constraints on the fixed point sets of certain groups acting on simplicial complexes.  Claim (1) from the next theorem is Theorem 2 of \cite{Smith:1941}, and Claim (2) is Lemma 1 of \cite{Oliver:1975}.  Claim (3) follows from the fact that under the given conditions every face in $\Delta \setminus \Delta^G$ lies in a $G$-orbit whose size is divisible by $p$.  A good description of the key ideas behind the theorem appears in \cite{Oliver:1975}.

\begin{theorem} \label{Smith}
Assume that the group $G$ acts on the finite simplicial complex $\Delta$ so that the set of fixed faces $\Delta^{H}$
of any subgroup $H \leq G$ forms a subcomplex of $\Delta$.
\begin{enumerate}
\item If $\Delta$ is $\mathbb{F}_{p}$-acyclic
and $G$ is a $p$-group, then $\Delta^{G}$ is $\mathbb{F}_{p}$-acyclic.
\item If $\Delta$ is $\mathbb{Q}$-acyclic
and $G$ is a cyclic group, then $\tilde{\chi}(\Delta^{G})=0$.
\item If $G$ is a $p$-group, then $\tilde{\chi}(\Delta)\equiv\tilde{\chi}(\Delta^{G})\mod p$.
\end{enumerate}
\end{theorem}

\begin{corollary} \label{Smithcor}
Assume that the group $G$ acts on the finite simplicial complex $\Delta$ so that the set of fixed faces $\Delta^{H}$
of any subgroup $H \leq G$ forms a subcomplex of $\Delta$.
\begin{enumerate}
\item[(A)] Assume further that $1 \unlhd P \unlhd N \unlhd G$ with $P$ a $p$-group, $N/P$ cyclic, and $G/N$ of prime power order.  If $\Delta$ is $\ff_p$-acyclic then $\Delta^G \neq \emptyset$.
\item[(B)] Assume now that $1 \unlhd C \unlhd G$ with $C$ cyclic and $G/C$ of prime power order. If $\Delta$ is $\qq$-acyclic then $\Delta^G \neq \emptyset$.
\end{enumerate}
\end{corollary}

Corollary~\ref{Smithcor} has been used before in combinatorics, although such use seems to be rare.  Claim (B) of the corollary will suffice for our purposes. Claim (A) was used (with $N=G$) in \cite{Shareshian/Woodroofe:2016}. Kahn, Saks, and Sturtevant earlier had used Claim (A) in \cite{Kahn/Saks/Sturtevant:1984} to attack Karp's Evasiveness Conjecture.  Claim (A) is proved by applying Claims (1)-(3) of Theorem~\ref{Smith} in order, and observing that an $\ff_p$-acyclic complex is $\qq$-acyclic.  Claim (B) follows from Claims (2) and (3) of the theorem.  

Finally, we state a special case of Claim (B) of Corollary~\ref{Smithcor} that suffices for our purposes here.  We observe that if a group $H$ acts in an order preserving manner on a poset $\p$, then $(\Delta \p)^H=\Delta(\p^H)$, hence Theorem~\ref{Smith} applies.

\begin{lemma} \label{smith}
Assume that the group $G$ acts in an order preserving manner on the finite poset $\p$.  If $\de \p$ has trivial reduced rational homology and $G$ is the direct product of a cyclic group and a group of prime power order, then the fixed point set $\p^G$ is nonempty.
\end{lemma}

\subsection{Proving Theorem~\ref{cosetposet}}The coset poset $\cg$ consists of all cosets $Hx$ such that $H$ is a proper subgroup of $G$.   Assume that $N \unlhd G$.  Following Brown (see \cite{Brown:2000}) we define the subposet of $\cg$,
$$
\cgn:=\{Hx \in \cg:HN=G\}.
$$
Brown shows in \cite[Proposition 9]{Brown:2000} that $\de\cg$ is homotopy equivalent to the simplicial join of $\de\cgn$ and $\de\cgq$. Using the K\"unneth formula for joins (see for example \cite[(9.12)]{Bjorner:1995}) and induction on the length of a chief series for $G$, we see that Theorem~\ref{cosetposet} follows directly from the next result.

\begin{proposition} \label{cgnprop}
If $N$ is a minimal normal subgroup of the finite group $G$, then $\de\cgn$ has nontrivial reduced rational homology.
\end{proposition}

The remainder of this section is devoted to proving Proposition~\ref{cgnprop}.

\begin{lemma} \label{abelian}
Proposition~\ref{cgnprop} is true under the assumption that $N$ is abelian.
\end{lemma}

\begin{proof}
If $N$ is an abelian minimal normal subgroup of $G$ and $HN=G$, then $H \cap N=1$, as $H \cap N$ is normalized by both $N$ and $H$.  It follows that $|H|=[G:N]$.  Therefore, $\cgn$ is an antichain of size divisible by $|N|$.  If $\cgn=\emptyset$, then $\de\cgn=\{\emptyset\}$ and $\widetilde{H}_{-1}(\de\cgn,\qq) \neq 0$.  Otherwise, $\de\cgn$ is not connected and $\widetilde{H}_0(\de\cgn,\qq) \neq 0$.
\end{proof}

We assume from now on that the minimal normal subgroup $N$ of $G$ is nonabelian.   


There is an action $G \times G$ on $\cg$, defined by
\begin{equation} \label{act}
Hx^{(g,k)}:=g^{-1}Hxk=H^gg^{-1}xk.
\end{equation}
As $N \unlhd G$, $\cgn$ is invariant under the action defined in (\ref{act}), and the resulting action of $G \times G$ on $\cgn$ restricts to an action of $C \times P$ on $\cgn$ whenever $C,P \leq G$.  With Lemma~\ref{smith} in mind, we are interested in the case where $C$ is cyclic and $P$ has prime power order.  In order to apply the lemma, we must understand fixed points in the given action.  These are described in \cite[Lemma 3.8]{Shareshian/Woodroofe:2016}, which is not hard to prove.

\begin{lemma} \label{actlem}
If $C \times P \leq G \times G$ acts on $\cg$ as in (\ref{act}), then $\cg^{C \times P}$ consists of those $Hx$ such that $\langle C,P^{x^-1} \rangle \leq H$
\end{lemma}

\begin{corollary} \label{fpfcor}
If $C \times P \leq N \times N$ acts on $\cg$ as in (\ref{act}) and $\langle C,P^g \rangle=N$ for all $g \in G$, then $\cgn^{C \times P}=\emptyset$.
\end{corollary}

\begin{proof}
Assume for contradiction that $Hx \in \cgn^{C \times P}$.  We know that $\langle C,P^{x^-1} \rangle \leq H$ by Lemma~\ref{actlem}.  Therefore, $N \leq H$.  This is impossible, as $H \neq G$ and $HN=G$.
\end{proof}

Proposition~\ref{cgnprop} will follow from Lemma~\ref{smith} and Corollary~\ref{fpfcor} if we can show that there exist cyclic $C \leq N$ and $P \leq N$ of prime power order satisfying $\langle C,P^g \rangle=N$ for all $g \in G$ whenever $G$ has a nonabelian minimal normal subgroup $N$.  Given such $N$, there exist a nonabelian simple group $L$ and subgroups $L_i$ ($1 \leq i \leq t$) of $N$ such that each $L_i$ is isomorphic with $L$ and $N$ is the internal direct product of the $L_i$.  Moreover, the $L_i$ are all the minimal normal subgroups of $N$ and are therefore permuted by the conjugation action of any $g \in G$. (All of this can be found for example in \cite[Section 4.3]{Dixon/Mortimer:1996}.)  We write $\pi_i$ for the projection of $N$ onto $L_i$ and observe that for $J \leq N$, $J \cap L_i$ is a normal subgroup of $\pi_i(J)$.  In what follows, we abuse notation by identifying each $L_i$ with $L$.

By Theorem~\ref{main1}, there exist $D,Q \leq L$ such that $D$ is cyclic, $Q$ is a Sylow $p$-subgroup of $L$ for some prime $p$, and $D$ and $Q$ together generate $L$ invariably. Pick $P \leq N$ such that $\pi_i(P)=P \cap L_i=Q$ for all $i \in [t]$.  So, $P$ is a $t$-fold direct product of copies of $Q$ and therefore a Sylow $p$-subgroup of $N$.  Pick $C \leq N$ such that $\pi_i(C)=D$ for all $i \in [t]$ and $C \cap \prod_{i \in I}L_i=1$ for all proper subsets $I$ of $[t]$.  Then $C \cong D$, hence $C$ is cyclic.

Assume $M \leq N$ contains $C$ and a $G$-conjugate of $P$, the latter being a Sylow $p$-subgroup of $N$ and therefore a direct product of Sylow $p$-subgroups of $L$.  Since $Q$ and $D$ generate $L$ invariably, $\pi_i(M)=L_i$ for all $i \in [t]$.  On the other hand, since $P \cap L_i \neq 1$ and $L_i \unlhd N$, we see that $L_i \cap M \neq 1$ for all $i$.  Since $L$ is simple and $M \cap L_i \unlhd \pi_i(M)$, we conclude that $M \cap L_i=L_i$ for all $i$, that is, $M=N$.  Therefore $\langle C,P^g \rangle=N$ for all $g \in G$.  Applying Corollary~\ref{fpfcor} and Lemma~\ref{smith}, we conclude that Proposition~\ref{cgnprop} holds.  Theorem~\ref{cosetposet} follows.

\section{Final Comments} \label{sec.final}

Given Proposition~\ref{Lietype}, one might ask whether every finite simple group is generated invariably by a Sylow subgroup and an element of prime order.  One can confirm using \cite{Conway/Curtis/Norton/Parker/Wilson:1985} that $M_{12}$ is generated invariably by a Sylow $3$-subgroup and an element of order eleven.  Combined with Proposition~\ref{sporadic}, this provides a positive answer for sporadic groups.  

The question remains open for alternating groups.  In fact, we do not know whether every alternating group is generated invariably by two Sylow subgroups.  Some positive results appear in \cite{Guralnick/Shareshian/Woodroofe:2023} and \cite{Shareshian/Woodroofe:2018}.  The strongest result in this direction is due to Ter\"{a}v\"{a}inen, who showed in \cite{Teravainen:2023} that the set of $n$ for which $A_n$ is generated invariably by two elements of prime order has asymptotic density one in the set of positive integers.

Several of our claims about sporadic groups and small groups of Lie type that we proved using ad hoc arguments were confirmed using GAP.  Our GAP code and its output will be available as arXiv ancillary files.

\bibliographystyle{alpha}
\bibliography{Russbib2}

\end{document}